\documentclass{article}
\usepackage{a4wide}
\usepackage[UKenglish]{babel}
\usepackage{graphicx}
\usepackage{float}

\usepackage[colorlinks=true,linkcolor=blue,citecolor=red]{hyperref}
\usepackage{amssymb,empheq,bbold}
\usepackage[inline]{enumitem}
\usepackage{geometry}
\usepackage{caption}
\usepackage{amsthm}
\usepackage{bbm}
\usepackage{tikz}

\newcommand \commentout[1] {}
\newcommand{\R}{\mathbb{R}}

\newtheorem{theorem}{Theorem}
\newtheorem{definition}[theorem]{Definition}
\newtheorem{proposition}[theorem]{Proposition}
\newtheorem{lemma}[theorem]{Lemma}
\theoremstyle{definition}

\title{A Nonlocal $p$-Laplacian Interface Model with Sharp Interface}

\author{
Kehan Shi\thanks{{Department of Mathematics, China Jiliang University, Hangzhou 310018, China.} \texttt{kshi@cjlu.edu.cn}}
\and
Zuoqiang Shi\thanks{{Yau Mathematical Sciences Center, Tsinghua University, Beijing 100084, China;}  {Yanqi Lake Beijing Institute of Mathematical Sciences and Applications, Beijing 101408, China.} \texttt{zqshi@tsinghua.edu.cn}}%
\and
Tangjun Wang\thanks{{Corresponding author. Department of Mathematics, The University of Hong Kong, Pokfulam Road, Hong Kong SAR, China.} \texttt{wangtj@hku.hk}}
}
\date{}

\begin{document}

\maketitle

\begin{abstract}
We propose an energy-based nonlocal $p$-Laplacian interface problem. Neumann interface conditions are naturally formulated via the energy, while Dirichlet conditions are enforced through a penalty term. A key feature is that the model retains a sharp interface, which facilitates extension to other interface problems; we illustrate this by developing a nonlocal approximation for the $p$-Laplacian interface problem with membrane conditions. By establishing $\Gamma$-convergence and compactness, we prove that as the nonlocal horizon vanishes, minimizers of the nonlocal functionals converge to those of the local counterparts. Numerical experiments using an efficient finite element method confirm the convergence.
\end{abstract}

\medskip


\medskip


\section{Introduction}
An interface problem concerns a partial differential equation (PDE) defined on a bounded domain $\Omega\subset \R^d$ that is divided by an interface $\Gamma\subset \Omega$, across which the solution or its derivatives may have discontinuities. Typically, one considers
\begin{equation}\label{eq:1.1}
  -\mbox{div} (\lambda(x) \nabla u) = f \quad \text{in } \Omega,
\end{equation}
subject to the homogeneous Neumann boundary condition $\nabla u\cdot\vec{n}=0$ on $\partial \Omega$, where the diffusion coefficient $\lambda(x)$ is piecewise smooth but may jump across the interface $\Gamma$. The solution $u$ satisfies interface conditions on $\Gamma$, such as continuity of $u$ and a prescribed jump of the normal flux. Namely,
\begin{equation}\label{eq:1.2}
[u]_\Gamma = 0 \quad \mbox{and}\quad
\quad \left[\lambda(x)\nabla u\cdot\vec{n}_\Gamma \right]_\Gamma = g,
\end{equation}
where the notation $[v]_\Gamma$ denotes the jump of function $v:\Omega\rightarrow\R$ across the interface $\Gamma$
and $\vec{n}_\Gamma$ is the unit normal vector on $\Gamma$.
The interface problem arises naturally in fluid dynamics, materials science, biochemistry, etc. \cite{anderson1998diffuse,li2006immersed,alali2015peridynamics}.

The purpose of this paper is to extend the interface problem \eqref{eq:1.1}--\eqref{eq:1.2} to the nonlocal setting.
Nonlocal equations provide a powerful framework for modeling processes with long-range interactions and memory effects that cannot be adequately described by local PDEs \cite{andreu2010nonlocal,du2019nonlocal}.
They have attracted significant attention in recent years and have been widely used in peridynamics \cite{rokkam2019nonlocal}, biological modeling \cite{painter2024biological}, image processing \cite{shi2021image}, etc.
The nonlocal counterpart of equation \eqref{eq:1.1} has been extensively investigated in the literature. However, extending interface problems to the nonlocal setting introduces additional challenges, as the local interface conditions in \eqref{eq:1.2} cannot be directly imposed within a nonlocal framework.
Among the few related studies \cite{alali2015peridynamics,capodaglio2020energy,seleson2013interface}, the interface condition is typically treated via a volume constraint approach, in which the interface $\Gamma$ is approximated by an extended interface $\Gamma_\delta=\{x\in\Omega: \mbox{dist}(x,\Gamma)<\delta\}$ of thickness $\delta$.

In this paper, we propose a new nonlocal approximation for the interface problem \eqref{eq:1.1}--\eqref{eq:1.2}. Specifically, we consider the associated energy functional
\[
\mathcal{E}(u)
=\frac{1}{2}\int_\Omega\lambda(x)|\nabla u|^2dx -\int_\Omega fudx-\int_\Gamma gudS,\quad u\in H^1(\Omega),
\]
and propose to approximate it by the nonlocal functional
\begin{align*}
  \mathcal{E}_\delta(u)=
    \frac{1}{2\delta^2}\int_{\Omega}\int_{\Omega}\Lambda(x,y)R_\delta(x-y)|u(y)-u(x)|^2&dydx
-\int_{\Omega}fudx
-\int_\Gamma g K_{\delta,1}\ast u dS \\
&+\frac{1}{\delta}
\int_\Gamma\left|K_{\delta,1}\ast u-K_{\delta,2}\ast u\right|^2dS,\quad u\in L^2(\Omega).
\end{align*}
Here
\[
\lambda(x)=\lambda_1\chi_{\Omega_1}(x)+\lambda_2\chi_{\Omega_2}(x),\quad
\Lambda(x,y)=\lambda_1\chi_{\Omega_1}(x)\chi_{\Omega_1}(y)
+\lambda_2\chi_{\Omega_2}(x)\chi_{\Omega_2}(y),
\]
with constants $\lambda_1,\lambda_2>0$, $\Omega_1$ and $\Omega_2$ are two disjoint subsets of $\Omega$ divided by $\Gamma$ (see Figure \ref{fig:1} for an illustration),
$R_\delta$ and $K_\delta$ are two nonlocal kernels, and
\[
\left(K_{\delta,i}\ast u\right)(x)
=\frac{1}{\int_{\Omega_i}K_\delta(x-y)dy}\int_{\Omega_i}K_\delta(x-y)u(y)dy,\quad i=1,2.
\]
We consider the energy functional $\mathcal{E}$ rather than the associated Euler-Lagrange equation \eqref{eq:1.1}--\eqref{eq:1.2} because the interface condition for the normal flux in \eqref{eq:1.2}, which is difficult to handle in the nonlocal setting, is implicitly contained in the functional.
The same approach has been used in \cite{capodaglio2020energy}.
To approximate the interface integral $\int_\Gamma gudS$ in $\mathcal{E}$, we utilize the nonlocal average along $\Gamma$.
For any $u\in L^2(\Omega)$, $K_{\delta,1}\ast u\in C^1$ whenever $K_\delta\in C^1$. Thus the interface integrals $\int_\Gamma$ in $\mathcal{E}_\delta$ are well-defined.
Formally, one has $ \int_\Gamma g K_{\delta,1}\ast u dS \approx \int_\Gamma gudS$
for small $\delta>0$.
The last term in $\mathcal{E}_\delta$ is motivated by \cite{wang2023nonlocal,gan2024convergence} that penalizes the jump of $u$ across $\Gamma$.
Namely, it approximates the first interface condition in \eqref{eq:1.2}.
This is not necessary in the local energy $\mathcal{E}$, as the continuity of $u$ across $\Gamma$ is achieved through the solution space $u\in H^1(\Omega)$.

In contrast to existing nonlocal interface models \cite{alali2015peridynamics,capodaglio2020energy,seleson2013interface}, the proposed nonlocal functional $\mathcal{E}_\delta$ is formulated with a sharp interface.
This yields several favorable properties.
First, it leads to independent nonlocal diffusion on $\Omega_1$ and $\Omega_2$, i.e.,
\begin{align*}
\int_\Omega&\int_\Omega \Lambda(x,y) R_\delta(x-y)|u(y)-u(x)|^2dydx\\
&=\lambda_1\int_{\Omega_1}\int_{\Omega_1}R_\delta(x-y)|u(y)-u(x)|^2dydx
+\lambda_2\int_{\Omega_2}\int_{\Omega_2}R_\delta(x-y)|u(y)-u(x)|^2dydx.
\end{align*}
Second, no extension of $g$ from $L^2(\Gamma)$ to $H^1(\Omega)$ is needed.
Finally, the approach can be easily generalized to more complicated interface problems.
To illustrate it, we consider a nonlinear interface problem of $p$-Laplacian type with membrane conditions
\begin{align}\label{eq:1.3}
\begin{cases}
-\mbox{div}(\lambda(x)|\nabla u|^{p-2}\nabla u)=f, & \mbox{in } \Omega, \\
|\nabla u|^{p-2}\nabla u\cdot\vec{n}=0, & \mbox{on } \partial\Omega, \\
\lambda_1 |\nabla u_1|^{p-2}\nabla u_1\cdot \vec{n}_\Gamma
=\gamma |u_2-u_1|^{p-2}(u_2-u_1)+g, & \mbox{on } \Gamma, \\
\lambda_2 |\nabla u_2|^{p-2}\nabla u_2\cdot \vec{n}_\Gamma
=\gamma |u_2-u_1|^{p-2}(u_2-u_1), & \mbox{on } \Gamma,
\end{cases}
\end{align}
where $1<p<\infty$, $\gamma>0$ is a constant, and
\[
u=
\begin{cases}
u_1, & \text{in } \Omega_1, \\
u_2, & \text{in } \Omega_2.
\end{cases}
\]
In equation \eqref{eq:1.3}, the normal flux across the interface $\Gamma$ depends on the jump of the solution.
If there is no jump of the solution through $\Gamma$ and $p=2$, equation \eqref{eq:1.3} reduces to equation \eqref{eq:1.1}--\eqref{eq:1.2}.
The linear case $p=2$ of \eqref{eq:1.3} has been studied extensively in recent years
\cite{dorfler2020elliptic,maia2023some,ciavolella2021existence,maia2024generalized}.
To further demonstrate the generality of the proposed approach, we focus on the $p$-Laplacian operator in \eqref{eq:1.3}.
Studies on interface problems involving $p$-Laplacian operators can be found in
\cite{gwinner2024optimal,borsuk2024interface}.

The associated energy functional for \eqref{eq:1.3} reads
\[
\mathcal{F}(u)
=\frac{1}{p}\int_{\Omega}\lambda(x)|\nabla u|^pdx
 -\int_\Omega fudx
 -\int_\Gamma gu_1dS
 +\frac{\gamma}{p}\int_\Gamma |u_1-u_2|^pdS,\quad u\in W^{1,p}(\Omega_1)\cap W^{1,p}(\Omega_2),
\]
and it is approximated by the nonlocal functional
\begin{align*}
  \mathcal{F}_\delta(u)=
    \frac{1}{p\delta^p}\int_{\Omega}\int_{\Omega}\Lambda(x,y)R_\delta(x-y)|u(y)-u(x)|^p&dydx
-\int_{\Omega}fudx
-\int_\Gamma g K_{\delta,1}\ast u dS \\
&+
\frac{\gamma}{p}\int_\Gamma\left|K_{\delta,1}\ast u-K_{\delta,2}\ast u\right|^pdS,\quad u\in L^p(\Omega).
\end{align*}
An interesting fact is that $\mathcal{E}_\delta$ and $\mathcal{F}_\delta$ have the same structure.
More precisely, $\mathcal{F}_\delta$ coincides with $\mathcal{E}_\delta$ by taking $p=2$ and $\gamma=\frac{2}{\delta}$.

Under standard assumptions, the existence and uniqueness of minimizers for $\mathcal{E}_\delta$ and $\mathcal{F}_\delta$ can be established via the direct method of the calculus of variations.
The results rely on the nonlocal Poincar\'{e} equality.
The functionals $\mathcal{E}_\delta$ and $\mathcal{F}_\delta$ are introduced as approximations of $\mathcal{E}$ and $\mathcal{F}$ respectively, and it is therefore necessary to justify the validity of these approximations. This is achieved by studying their $\Gamma$-convergence and compactness.

In addition to the theoretical analysis, we also investigate numerical solutions of the nonlocal model. In this work, we adopt the method proposed in \cite{gan2025nonlocal}, which exploits the tensor-product structure of the Gaussian kernel and the multi-cubic polynomial basis to decompose a $2d$-dimensional integral into the product of two $d$-dimensional integrals, thereby significantly reducing the computational cost. With the aid of this efficient numerical solver, we verify that the minimizer of the nonlocal model converges to that of the corresponding local model at a first-order rate.

This paper is organized as follows. In section 2, we present the assumptions and mathematical tools needed for the study.
The existence and uniqueness of minimizers for the functionals are proved in section 3. Section 4 is devoted to the proof of $\Gamma$-convergence and compactness of the functionals, which lead to the convergence of the minimizers. Numerical methods and experiments are presented in section 5. We conclude the paper in section 6.

\section{Preliminaries}

\subsection{Settings}

Let $\Omega\subset\mathbb{R}^d$ be a bounded domain with Lipschitz boundary $\partial\Omega$,
$\Omega_1$ and $\Omega_2$ be two disjoint subsets of $\Omega$ such that $\partial\Omega_1\subset \partial\Omega_2$ and $\Omega=\overline{\Omega}_1\cup\Omega_2$.
The interface, i.e., the common boundary between $\Omega_1$ and $\Omega_2$, is defined as
$\Gamma=\overline{\Omega}_1\cap \overline{\Omega}_2$.
The unit outward normal vectors on $\partial\Omega$ and $\partial\Omega_1$ are denoted by $\vec{n}$ and $\vec{n}_\Gamma$, respectively.
See Figure \ref{fig:1} for an illustration.

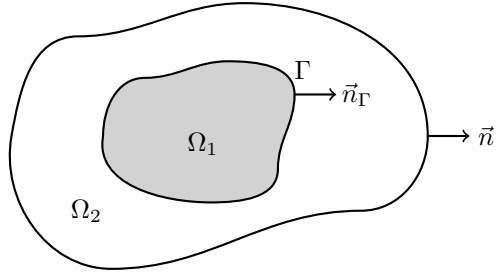
\begin{figure}[H]
  \centering
\begin{tikzpicture}[scale=1.1]
\begin{scope}
\filldraw[thick, fill=white]
  (-2,0)
  to[out=80,  in=180] (-1.2,1.2)
  to[out=0,   in=180] (0.8,1.6)
  to[out=0,   in=90] (3,0)
  to[out=270, in=0]   (2.2,-0.9)
  to[out=180, in=360] (-0.8,-1.6)
  to[out=180, in=260] (-2,0);

\draw[->,thick]
  (3,0) -- (3.5,0);
\node at (3.7,0) {$\vec{n}$};
\end{scope}

\begin{scope}
\filldraw[thick, fill=gray!35]
  (-0.9,0)
  to[out=90,  in=180] (-0.4,0.7)
  to[out=0,   in=180] (0.6,0.9)
  to[out=0,   in=90] (1.4,0.5)
  to[out=270, in=90]  (1.2,-0.4)
  to[out=270, in=360] (0.4,-0.8)
  to[out=180, in=260] (-0.9,0);

\node at (0.3,-0.1) {$\Omega_1$};

\draw[->,thick]
  (1.4,0.5) -- (1.9,0.5);
\node at (2.15,0.5) {$\vec{n}_\Gamma$};
\end{scope}

\node at (-1.1,-0.9) {$\Omega_2$};
\node at (1.5,0.8) {$\Gamma$};
\end{tikzpicture}
\caption{Illustration of the domains $\Omega_1$, $\Omega_2$, and the interface $\Gamma$.}
\label{fig:1}
\end{figure}

We are interested in the functional associated to the Poisson interface problem \eqref{eq:1.1}--\eqref{eq:1.2}
\begin{align}\label{eq:energy_local}
\mathcal{E}(u)=
\begin{cases}
  \frac{1}{2}\int_{\Omega}\lambda(x)|\nabla u|^2dx
  -\int_{\Omega}fudx -\int_\Gamma gudS, &\mbox{ if } u\in\mathcal{U}(\Omega),\\
  +\infty, &\mbox{ otherwise},
\end{cases}
\end{align}
where $\lambda(x)=\lambda_1\chi_{\Omega_1}(x)+\lambda_2\chi_{\Omega_2}(x)$,
$\lambda_1,\lambda_2$ are positive constants, $f\in L^2(\Omega)$, $g\in L^2(\Gamma)$, and
\[
\mathcal{U}(\Omega)=\left\{u\in H^1(\Omega): \frac{\partial u}{\partial\vec{n}}=0 \mbox{ on } \partial\Omega, \int_{\Omega} udx=0\right\}.
\]
The nonlocal analogue of $\mathcal{E}$ reads
\begin{align}\label{eq:energy_nonlocal}
\begin{split}
  \mathcal{E}_\delta(u)=
  \begin{cases}
     \frac{1}{2\delta^2}\int_\Omega\int_\Omega \Lambda(x,y) R_\delta(x-y)|u(y)-u(x)|^2dydx-\int_{\Omega}fudx\\
\qquad\qquad\qquad\quad
- \int_\Gamma g K_{\delta,1}\ast u dS
+\frac{1}{\delta}
\int_\Gamma\left|K_{\delta,1}\ast u-K_{\delta,2}\ast u\right|^2dS,
\quad\mbox{if } u\in L^2_0(\Omega), \\
 +\infty,  \qquad\qquad\qquad\qquad\qquad\qquad\qquad\qquad\qquad\qquad\qquad
    \qquad\qquad\quad\mbox{otherwise},
  \end{cases}
\end{split}
\end{align}
where $\Lambda(x,y)=\lambda_1\chi_{\Omega_1}(x)\chi_{\Omega_1}(y)
+\lambda_2\chi_{\Omega_2}(x)\chi_{\Omega_2}(y)$,
\[
L^p_0(\Omega)=\left\{u\in L^p(\Omega): \int_{\Omega}udx=0\right\}, \quad 1<p<\infty,
\]
and for any $\delta>0$,
\[
R_\delta(s)=\frac{1}{\delta^d}R\left(\frac{|s|^2}{\delta^2}\right),\quad
K_\delta(s)=\frac{1}{\delta^d}K\left(\frac{|s|^2}{\delta^2}\right),
\quad s\in\R,
\]
the nonlocal kernels $R, K: [0,\infty)\rightarrow [0,\infty)$ are nonincreasing $C^1$ functions with compact support.
Assume w.l.o.g. that
\[
  \int_{\R^d}R(|z|^2)|z_1|^2dz=1,
\]
where $z_1$ is the first coordinate of $z$.
Here and in the following, we use the normalized convolution notation
\[
\left(K_{\delta,i}\ast u\right)(x)
=\frac{1}{\int_{\Omega_i}K_\delta(x-y)dy}\int_{\Omega_i}K_\delta(x-y)u(y)dy,\quad i=1,2.
\]

We further study the functional associated to the $p$-Laplacian interface problem with membrane
conditions \eqref{eq:1.3}
\begin{align*}
\mathcal{F}(u)=
\begin{cases}
  \frac{1}{p}\int_{\Omega}\lambda(x)|\nabla u|^pdx
 -\int_\Omega fudx
 -\int_\Gamma gu_1dS
 +\frac{\gamma}{p}\int_\Gamma |u_1-u_2|^pdS, &\mbox{ if } u\in \mathcal{W}(\Omega),\\
  +\infty, &\mbox{ otherwise},
\end{cases}
\end{align*}
where $1<p<\infty$, $\gamma>0$ is a constant, $f\in L^q(\Omega)$, $g\in L^q(\Gamma)$, $q=\frac{p}{p-1}$,
\[
u=
\begin{cases}
u_1, & \text{in } \Omega_1, \\
u_2, & \text{in } \Omega_2,
\end{cases}
\]
and
\[
\mathcal{W}(\Omega)=\left\{u\in W^{1,p}(\Omega_1)\cap W^{1,p}(\Omega_2):
\frac{\partial u}{\partial\vec{n}}=0 \mbox{ on } \partial\Omega, \int_{\Omega} udx=0\right\}.
\]
The nonlocal analogue of $\mathcal{F}$ reads
\begin{align*}
\begin{split}
  \mathcal{F}_\delta(u)=
  \begin{cases}
     \frac{1}{p\delta^p}\int_{\Omega}\int_{\Omega}\Lambda(x,y)R_\delta(x-y)|u(y)-u(x)|^pdydx
-\int_{\Omega}fudx\\
\qquad\qquad\qquad\quad
-\int_\Gamma g K_{\delta,1}\ast u dS
+\frac{\gamma}{p}\int_\Gamma\left|K_{\delta,1}\ast u-K_{\delta,2}\ast u\right|^pdS,
\quad\mbox{if } u\in L^p_0(\Omega), \\
 +\infty,  \qquad\qquad\qquad\qquad\qquad\qquad\qquad\qquad\qquad\qquad\qquad
    \qquad\qquad\quad\mbox{otherwise}.
  \end{cases}
\end{split}
\end{align*}
Throughout this paper, $C$ stands for a positive constant, which may change from line to line.

\subsection{Mathematical tools}
As we study variational convergence for sequences of functionals, the mathematical tool we use is $\Gamma$-convergence,
which is introduced by De Giorgi \cite{dal2012introduction,braides2002gamma}.

\begin{definition}
Given a metric space $X$ and functionals $F_n, F: X\rightarrow [-\infty,\infty]$, it is said that $F_n$ $\Gamma$-converges to $F$ as $n\rightarrow\infty$, denoted by $F_n\stackrel{\Gamma}{\longrightarrow }F$,
if for every $x\in X$, the following holds.
\begin{itemize}
  \item Liminf inequality: For every sequence $\{x_n\}_{n\in\mathbb{N}}\subset X$ converging to $x$,
      \begin{equation*}
        \liminf_{n\rightarrow\infty} F_n(x_n)\geq F(x).
      \end{equation*}
  \item Limsup inequality: There exists a sequence $\{x_n\}_{n\in\mathbb{N}}\subset X$ converging to $x$ such that
      \begin{equation*}
        \limsup_{n\rightarrow\infty} F_n(x_n)\leq F(x).
      \end{equation*}
\end{itemize}
\end{definition}

A fundamental property of $\Gamma$-convergence is that, under the assumption of compactness, it implies the convergence of minimizers. See \cite[Theorem 1.21]{braides2002gamma} for the proof.

\begin{proposition}\label{pr:2.2}
  Given a metric space $X$ and functionals $F_n, F: X\rightarrow [-\infty,\infty]$, such that
  $F_n\stackrel{\Gamma}{\longrightarrow }F\not\equiv\infty$ as $n\rightarrow\infty$.
  If there exists a precompact sequence $\{x_n\}_{n\in\mathbb{N}}$ such that
  \begin{equation*}
    \lim_{n\rightarrow\infty}\left(F_n(x_n)-\inf_{x\in X}F_n(x)\right)=0,
  \end{equation*}
  then
  \begin{equation*}
    \lim_{n\rightarrow\infty}\inf_{x\in X}F_n(x)=\inf_{x\in X}F(x),
  \end{equation*}
  and any cluster point of $\{x_n\}_{n\in\mathbb{N}}$ is a minimizer of $F$.
\end{proposition}

We also require two standard results introduced for the study of nonlocal $p$-Laplacian operators, i.e., Poincar\'{e}'s inequality and compactness.
The proof can be founded in \cite[Proposition 4.1, Proposition 3.2]{andreu2008nonlocal}.

\begin{proposition}\label{pr:2.3}
Let $u\in L^p(\Omega)$, $1\leq p<\infty$, and $\delta>0$ be fixed. There exists a constant $C>0$ depending only on $R$, $\Omega$, $p$, and $\delta$, such that
\begin{equation*}
  C\int_\Omega\left|u-\frac{1}{|\Omega|}\int_\Omega udx\right|^p
  \leq \int_\Omega\int_\Omega R_\delta(x-y)|u(y)-u(x)|^pdydx.
\end{equation*}
\end{proposition}
In the following, we call $\delta_n$ a null sequence if $\delta_n\rightarrow 0$ as $n\rightarrow \infty$.
\begin{proposition}\label{pr:2.4}
Let $\delta_n$ be a null sequence, $u_n, u\in L^p(\Omega)$, $1< p<\infty$, such that $u_n\rightharpoonup u$ in $L^p(\Omega)$ and
\[
\sup_{n}\frac{1}{\delta^p_n}\int_\Omega\int_\Omega R_{\delta_n}(x-y)|u_n(y)-u_n(x)|^pdydx<\infty.
\]
Then we have $u\in W^{1,p}(\Omega)$ and
\[
R(|z|)^{1/p}\chi_{\Omega}\left(x+\delta_n z\right)\frac{u_n(x+\delta_n z)-u_n(x)}{\delta_n}
\rightharpoonup
R(|z|^2)^{1/p} z\cdot\nabla u,
\]
in $L^p(\Omega)\times L^p(\R^d)$.
If further assume that $\Omega_1$, $\Omega_2$ are smooth, then $\{u_n\}$ is precompact in $L^p(\Omega)$.

\end{proposition}

\section{Existence and uniqueness of minimizers}
In this section, we study the existence and uniqueness of minimizers for $\mathcal{E}_\delta$ and $\mathcal{F}_\delta$.
Some observations for $K_{\delta,i}\ast u$ are needed.
\begin{lemma}\label{le:Ku}
  Let $i=1,2$, $1< p<\infty$, $u_n, u\in L^p(\Omega)$, and $u_n\rightharpoonup u$ in $L^p(\Omega)$.
  We have, for any $\delta>0$,
  \begin{equation}\label{eq:Ku:3}
    K_{\delta,i}\ast u_n\rightarrow  K_{\delta,i}\ast u,\quad \mbox{in } L^p(\Gamma),
  \end{equation}
  as $n\rightarrow\infty$.
  Besides,
  \begin{equation}\label{eq:Ku:4}
    \int_{\Omega_i} |\nabla K_{\delta,i}\ast u|^pdx
    \leq
    \frac{C}{\delta^p}\int_{\Omega_i}\int_{\Omega_i}R_\delta(x-y)|u(y)-u(x)|^pdydx,
  \end{equation}
  where the constant $C$ depends not on $\delta$.

  If further $u_n\rightarrow u\in W^{1,p}(\Omega)$ in $L^p(\Omega)$, $\delta_n$ is a null sequence, and
  \[
  \sup_{n}\frac{1}{\delta_n^p}\int_{\Omega_i}\int_{\Omega_i}R_{\delta_n}(x-y)|u_n(y)-u_n(x)|^pdydx<\infty,
  \]
  we have
  \begin{equation}\label{eq:Ku:5}
    K_{\delta_n,i}\ast u_n\rightarrow u,\quad \mbox{in } L^p(\Gamma),
  \end{equation}
  as $n\rightarrow\infty$.
  \end{lemma}

\begin{proof}
Let $\delta>0$ be fixed.
For any $x\in\Gamma$,
since $u_n\rightharpoonup u$ in $L^p(\Omega)$,
\[
K_{\delta,i}\ast u_n(x)
=\frac{\int_{\Omega_i}K_\delta(x-y)u_n(y)dy}{\int_{\Omega_i}K_\delta(x-y)dy}
\rightarrow \frac{\int_{\Omega_i}K_\delta(x-y)u(y)dy}{\int_{\Omega_i}K_\delta(x-y)dy}
=K_{\delta,i}\ast u(x),
\]
as $n\rightarrow\infty$.
Observe that
\[
|K_{\delta,i}\ast u_n(x)|
\leq C\left|\int_{\Omega_i}K_\delta(x-y)u_n(y)dy\right|
\leq C\int_{\Omega_i}\left|u_n(y)\right|dy
\leq C.
\]
Thus \eqref{eq:Ku:3} follows from Lebesgue's dominated convergence theorem.

\eqref{eq:Ku:4} has been proven in \cite[Lemma 5.4]{gan2024convergence}. We omit it here.
\eqref{eq:Ku:5} is a corollary of \eqref{eq:Ku:4}.
In fact,
\[
\int_{\Omega_i}|K_{\delta_n,i}\ast u_n|^pdx
\leq C\int_{\Omega_i}\left|\int_{\Omega_i}K_{\delta_n}(x-y)u_n(y)dy\right|^pdx
\leq C\int_{\Omega_i}|u_n|^pdx.
\]
This together with \eqref{eq:Ku:4} imply that
$K_{\delta_n,i}\ast u_n$ is uniformly bounded in $W^{1,p}(\Omega_i)$ and admits a weakly convergent subsequences in $W^{1,p}(\Omega_i)$.
Besides,
\begin{align*}
  \|K_{\delta_n,i}\ast u_n-u\|_{L^p(\Omega_i)}
  \leq \|K_{\delta_n,i}\ast u_n-K_{\delta_n,i}\ast u\|_{L^p(\Omega_i)}
  +\left\|K_{\delta_n,i}\ast u-\frac{\int_{\R^d}K_{\delta_n}(x-y)dy}{\int_{\Omega_i}K_{\delta_n}(x-y)dy}u\right\|_{L^p(\Omega_i)}\\
  +\left\|\frac{\int_{\R^d}K_{\delta_n}(x-y)dy}{\int_{\Omega_i}K_{\delta_n}(x-y)dy}u-u\right\|_{L^p(\Omega_i)}
  =:I_1+I_2+I_3,
\end{align*}
where $I_1\rightarrow 0$ due to $u_n\rightarrow u$ in $L^p(\Omega_i)$,
$I_2\rightarrow 0$ due to the classical result on mollifier,
and $I_3\rightarrow 0$ due to Lebesgue’s dominated convergence theorem.
Consequently,
\[
K_{\delta_n,i}\ast u_n\rightharpoonup u,\quad \mbox{in } W^{1,p}(\Omega_i).
\]
Then \eqref{eq:Ku:5} follow from Sobolev's embedding theorem and the trace inequality.
\end{proof}

A key tool for the existence and uniqueness of minimizers for $\mathcal{E}_\delta$ and $\mathcal{F}_\delta$ is the nonlocal Poincar\'{e} inequality of the following form.
It is a corollary of Proposition \ref{pr:2.3}.
\begin{lemma}\label{le:poincare}
  Let $u\in L^p(\Omega)$, $1\leq p<\infty$, and $\delta>0$ be fixed. There exists a constant $C>0$ depending only on $R$, $\Omega_1$, $\Omega_2$, $p$, and $\delta$, such that
  \begin{align}\label{eq:poincare}
  \begin{split}
    \frac{1}{C}\int_{\Omega} |u|^pdx
    \leq \int_{\Omega_1}\int_{\Omega_1}R_\delta(x-y)|u(y)-u(x)|^pdydx
    &+\int_{\Omega_2}\int_{\Omega_2}R_\delta(x-y)|u(y)-u(x)|^pdydx\\
    &+\int_\Gamma\left|K_{\delta,1}\ast u-K_{\delta,2}\ast u\right|^pdS
    +\left|\int_{\Omega} udx\right|^p.
    \end{split}
  \end{align}
\end{lemma}

\begin{proof}
  We prove it by contradiction. Assume that \eqref{eq:poincare} does not hold.
  Then for any $n\in\mathbb{N}$, there exists  $u_{n}\in L^p(\Omega)$, such that
  \begin{align*}
    \int_\Omega & |u_n|^pdx\\
    \geq &n
    \left(\int_{\Omega_1}\int_{\Omega_1}R_\delta(x-y)|u_{n}(y)-u_{n}(x)|^pdydx
    +\int_{\Omega_2}\int_{\Omega_2}R_\delta(x-y)|u_{n}(y)-u_{n}(x)|^pdydx\right.\\
    &\qquad\qquad\qquad\qquad\qquad\qquad\qquad\qquad
    \left.+\int_\Gamma\left|K_{\delta,1}\ast u_{n}-K_{\delta,2}\ast u_{n}\right|^pdS
    +\left|\int_{\Omega} u_ndx\right|^p \right),
  \end{align*}
  Let $v_{n}=\frac{1}{\|u_n\|_{L^p(\Omega)}} u_{n}$.
  We have $\int_{\Omega}|v_{n}|^pdx=1$
  and
    \begin{align*}
    \frac{1}{n}
    \geq
    \int_{\Omega_1}\int_{\Omega_1}R_\delta(x-y)|v_{n}(y)-v_{n}(x)|^pdydx
    &+\int_{\Omega_2}\int_{\Omega_2}R_\delta(x-y)|v_{n}(y)-v_{n}(x)|^pdydx\\
    &+\int_\Gamma\left|K_{\delta,1}\ast v_{n}-K_{\delta,2}\ast v_{n}\right|^pdS
    +\left|\int_{\Omega} v_ndx\right|^p.
  \end{align*}
  We find a subsequence of $\{v_{n}\}$ (still dented by itself) and a function $v\in L^p(\Omega)$ such that
  \begin{equation}\label{eq:le:poincare:1}
    v_{n} \rightharpoonup v, \quad \mbox{in } L^p(\Omega),
  \end{equation}
  and
  \begin{subequations}\label{eq:le:poincare:2}
  \begin{align}
    &\lim_{n\rightarrow\infty} \int_{\Omega_i}\int_{\Omega_i}R_\delta(x-y)|v_{n}(y)-v_{n}(x)|^pdydx=0,\quad i=1,2,\label{eq:le:poincare:2a}\\
    &\lim_{n\rightarrow\infty}
    \int_\Gamma\left|K_{\delta,1}\ast v_{n}-K_{\delta,2}\ast v_{n}\right|^pdS=0,\label{eq:le:poincare:2c}\\
    &\lim_{n\rightarrow\infty}
    \int_{\Omega} v_{n}dx=0. \label{eq:le:poincare:2d}
  \end{align}
  \end{subequations}
  It follows from \eqref{eq:le:poincare:1}, \eqref{eq:le:poincare:2a}, the weak lower semi-continuity, and Proposition \ref{pr:2.3} that
  \begin{align*}
    \int_{\Omega_i}\left|v-\frac{1}{|\Omega_i|}\int_{\Omega_i}vdx\right|^pdx
    &\leq \liminf_{n\rightarrow\infty}
    \int_{\Omega_i}\left|v_{n}-\frac{1}{|\Omega_i|}\int_{\Omega_i}v_{n}dx\right|^pdx\\
    &\leq \liminf_{n\rightarrow\infty} C \int_{\Omega_i}\int_{\Omega_i}R_\delta(x-y)|v_{n}(y)-v_{n}(x)|^pdydx=0,
  \end{align*}
  from which we have $v=c_i$ in $\Omega_i$ for constant $c_i$, $i=1,2$.
  Consequently, \eqref{eq:Ku:3} implies
    \[
  K_{\delta,1}\ast v_{n}-K_{\delta,2}\ast v_{n}\rightarrow c_1-c_2, \quad \mbox{in } L^p(\Gamma).
  \]
Combining it with \eqref{eq:le:poincare:2c}, we see that $c_1=c_2$.
  Passing to the limit in \eqref{eq:le:poincare:2d}, we obtain that $c_1=c_2=0$.
  This contradicts $\int_{\Omega}|v_{n}|^pdx=1$.
\end{proof}

With nonlocal Poincar\'{e}'s inequality \eqref{eq:poincare}, we are able to show  the coercivity of functional $\mathcal{E}_\delta$.
Then the existence of minimizers follows.
\begin{theorem}\label{th:existence_nonlocal}
  Let $\lambda_1, \lambda_2>0$, $f\in L^2(\Omega)$, and $g\in L^2(\Gamma)$. For any $\delta>0$, the nonlocal functional $\mathcal{E}_\delta$ admits a unique minimizer in $L^2_0(\Omega)$.
\end{theorem}

\begin{proof}
  Let $u\in L^2_0(\Omega)$.
  The definition of $\mathcal{E}_\delta$ and nonlocal Poincar\'{e}'s inequality \eqref{eq:poincare} imply
  \begin{align*}
    \int_\Omega |u|^2dx
    \leq C\mathcal{E}_\delta(u)
    +C\int_{\Omega}fudx + C\int_\Gamma g K_{\delta,1}\ast u dS.
  \end{align*}
  Utilizing Cauchy's inequality and the fact
  $\int_\Gamma |K_{\delta,1}\ast u|^2dS\leq C\int_{\Omega_1}|u|^2dx$,
  we have
  \begin{align*}
  \int_{\Omega}fudx + \int_\Gamma gK_{\delta,1}\ast u dS
  \leq& \frac{4}{\epsilon}\int_\Omega |f|^2dx
  +\epsilon\int_\Omega |u|^2dx
  +\frac{4}{\epsilon}\int_\Gamma |g|^2dS
  +C\epsilon \int_{\Omega_1}|u|^2dx.
  \end{align*}
  Combining two inequalities and choosing a small $\epsilon$ yield the coercivity of $\mathcal{E}_\delta$, i.e.,
    \begin{align}\label{eq:coercivity}
  \begin{split}
    \int_\Omega|u|^2dx
    \leq C\mathcal{E}_\delta(u)+C\int_\Omega |f|^2dx + C\int_\Gamma |g|^2dS.
  \end{split}
  \end{align}
There exists a minimizing sequence $u_k\in L_0^2(\Omega)$, such that
\[
\lim_{k\rightarrow\infty}\mathcal{E}_\delta(u_k)=\inf_{L^2(\Omega)} \mathcal{E}_\delta(u),
\]
and $\mathcal{E}_\delta(u_k)$ is uniformly bounded in $L^2(\Omega)$.
\eqref{eq:coercivity} further implies the existence of a subsequence of $\{u_k\}$ (still denoted by itself) and a function $u\in L^2(\Omega)$, such that
\begin{align*}
  u_{k} \rightharpoonup {u}, \quad \mbox{in } L^2(\Omega).
\end{align*}
By \eqref{eq:Ku:3},
  \[
  K_{\delta,i}\ast u_k\rightarrow K_{\delta,i}\ast u, \quad \mbox{in } L^2(\Gamma),
  \]
  as $k\rightarrow \infty$.
 By the weak lower semi-continuity,
\[
\mathcal{E}_\delta(u)
\leq \liminf_{k\rightarrow\infty}\mathcal{E}_\delta(u_k).
\]
Consequently,
\[
\inf_{L^2(\Omega)} \mathcal{E}_\delta(u)
\leq \mathcal{E}_\delta(u)
\leq \liminf_{k\rightarrow\infty}\mathcal{E}_\delta(u_k)
=\inf_{L^2(\Omega)} \mathcal{E}_\delta(u).
\]
Thus $u$ is a minimizer of $\mathcal{E}_\delta$ in $L^2(\Omega)$.
Clearly, $u\in L^2_0(\Omega)$.

To prove the uniqueness of the minimizer, we only need to show that
\[
E_\delta(u)=
\frac{1}{\delta^2}\sum_{i=1,2}\int_{\Omega_i}\int_{\Omega_i}\Lambda(x,y) R_\delta(x-y)|u(y)-u(x)|^2dydx
+\frac{1}{\delta}
\int_\Gamma\left|K_{\delta,1}\ast u-K_{\delta,2}\ast u\right|^2dS
\]
is strictly convex on $L^2_0(\Omega)$.
In fact, for any $u,v\in L^2_0(\Omega)$, $w:=u-v$, and $0<t<1$,
it follows from the elementary equality
\[
t|x|^2+(1-t)|y|^2-|tx+(1-t)y|^2=t(1-t)|x-y|^2,\quad x,y\in\R,
\]
and nonlocal Poincar\'{e}'s inequality \eqref{eq:poincare}
that
\begin{align*}
tE_\delta(u)+(1-t)E_\delta(v)-E_\delta(tu+(1-t)v)
=t(1-t)E_\delta(w)
\geq Ct(1-t)\int_\Omega |w|^2dx>0,
\end{align*}
if $u\neq v$.
This completes the proof.
\end{proof}

The proof also works for the nonlocal functional $\mathcal{F}_\delta$.
The only difference is that we use the elementary equality
\begin{align*}
  (|x|^{p-2}x-|y|^{p-2}y)(x-y)\geq
  \begin{cases}
    2^{2-p}|x-y|^p, & \mbox{if } p\geq 2, \\
    (p-1)(1+|x|^2+|y|^2)^{\frac{p-2}{2}}|x-y|^2, & \mbox{if } 1\leq p\leq 2,
  \end{cases}
\end{align*}
where $x,y\in \R$,
for the uniqueness of minimizers.

\begin{theorem}
  Let $1<p<\infty$, $q=\frac{p}{p-1}$, $\lambda_1, \lambda_2>0$, $f\in L^q(\Omega)$, and $g\in L^q(\Gamma)$. For any $\delta>0$, the nonlocal functional $\mathcal{F}_\delta$ admits a unique minimizer in $L^p_0(\Omega)$.
\end{theorem}

\section{Convergence of minimizers}
In this section, we prove the convergence of minimizers for $\mathcal{E}_{\delta_n}$ and $\mathcal{F}_{\delta_n}$ as $n\rightarrow\infty$.
This is achieved via the $\Gamma$-convergence and compactness of nonlocal functionals.
The associated results for nonlocal $p$-Laplacians have been well-studied \cite{slepcev2019analysis}.
Here, we only need to focus on the terms in $\mathcal{E}_{\delta_n}$ and $\mathcal{F}_{\delta_n}$ that involve the interface $\Gamma$.
We begin with the $\Gamma$-convergence of $\mathcal{E}_{\delta_n}$.

\begin{theorem}\label{th:Gamma_converge_delta}
  Let $\lambda_1, \lambda_2>0$, $f\in L^2(\Omega)$, and $g\in L^2(\Gamma)$. Then for any null sequence $\delta_n$,
  \begin{equation*}
    \mathcal{E}_{\delta_n}\stackrel{\Gamma}{\longrightarrow }\mathcal{E},
  \end{equation*}
  in $L^2(\Omega)$ as $n\rightarrow\infty$.
\end{theorem}
\begin{proof}
We prove the liminf inequality and the limsup inequality separately in the following.

The liminf inequality:
If $u_n\rightarrow u$ in $L^2(\Omega)$ as $n\rightarrow \infty$, we have
\[
 \liminf_{n\rightarrow \infty} \mathcal{E}_{\delta_n}(u_n)\geq \mathcal{E}(u).
\]
  Assume w.l.o.g. that $\mathcal{E}_{\delta_n}(u_n)<\infty$.
  Notice from Cauchy's inequality, the trace inequality, and \eqref{eq:Ku:4} that
  \begin{align*}
    &\frac{1}{\delta_n^2}\int_{\Omega}\int_{\Omega}\Lambda(x,y) R_{\delta_n}(x-y)|u_{n}(y)-u_{n}(x)|^2dydx
    \leq \mathcal{E}_{\delta_n}(u_n)
    +\int_{\Omega}fu_ndx
    +\int_\Gamma g K_{\delta_n,1}\ast u_n dS\\
    &\leq \mathcal{E}_{\delta_n}(u_n)+\int_\Omega |f|^2dx+C\int_\Omega |u_n|^2dx
    +\frac{1}{\epsilon}\int_\Gamma |g|^2dS\\
    &\qquad\qquad\qquad\qquad\qquad\qquad\qquad\qquad\qquad
    +\frac{C\epsilon}{\delta_n^2}\int_{\Omega_1}\int_{\Omega_1} R_{\delta_n}(x-y)|u_{n}(y)-u_{n}(x)|^2dydx.
  \end{align*}
  Choosing a small $\epsilon$, we arrive at
  \[
  \sup_{n}\frac{1}{\delta_n^2}\sum_{i=1,2}\int_{\Omega_i}\int_{\Omega_i} R_{\delta_n}(x-y)|u_{n}(y)-u_{n}(x)|^2dydx<\infty.
  \]
  Applying Proposition \ref{pr:2.4}, we have
  $u\in H^1(\Omega_1)\cap H^1(\Omega_2)$.
  Besides, by \eqref{eq:Ku:5},
    \begin{align*}
  \int_\Gamma \left|u|_{\Omega_1}-u|_{\Omega_2}\right|^2dS
  &= \lim_{n\rightarrow \infty}
 \int_\Gamma\left|K_{\delta_n,1}\ast{u}_{n}-K_{\delta_n,2}\ast{u}_{n}\right|^2dS \\
 &\leq \lim_{n\rightarrow \infty}
 \delta \left(\mathcal{E}_{\delta_n} (u_n) +\int_{\Omega}fu_n dx
 + \int_\Gamma gK_{\delta_n,1}\ast{u}_{n} dS\right)
 =0.
  \end{align*}
  Consequently, $u|_{\Omega_1}-u|_{\Omega_2}=0$ on $\Gamma$ and $u\in \mathcal{U}(\Omega)$.

  By the result on the $\Gamma$-convergence of nonlocal Laplacian \cite[Lemma 4.6]{slepcev2019analysis}, we have
  \[
  \liminf_{n\rightarrow \infty}
  \int_{\Omega_i}\int_{\Omega_i}R_{\delta_n}(x-y)|u_{n}(y)-u_{n}(x)|^2dydx
  \geq \int_{\Omega_i}|\nabla u|^2dx,\quad i=1,2.
  \]
  To pass to the limit for terms involving $K_{\delta_n,i}\ast{u}_{n}$, we utilize \eqref{eq:Ku:5}.
  This proves the liminf inequality.

The limsup inequality:
For any $u\in L^2(\Omega)$, there exists a sequence of functions $\{u_n\}\subset L^2(\Omega)$ such that $u_n\rightarrow u$ in $L^2(\Omega)$ as $n\rightarrow \infty$ and
  \[
    \limsup_{n\rightarrow 0} \mathcal{E}_{\delta_n}(u_n)\leq  \mathcal{E}(u).
  \]
Assume w.l.o.g. that $\mathcal{E}(u)<\infty$, i.e., $u\in H^1(\Omega)$. 
 By \cite[Remark 2.7]{garcia2016continuum},
  we only need to prove the limsup inequality for all $u$ in a dense subset of $H^1(\Omega)$, e.g., $C^2(\overline{\Omega})\cap H^1(\Omega)$.

  Taylor's expansion yields
  \[
  \lim_{n\rightarrow \infty}
  \frac{1}{\delta_n^2}\int_{\Omega_i}\int_{\Omega_i}R_{\delta_n}(x-y)|u_{i}(y)-u_{i}(x)|^2dydx
  \leq\int_{\Omega_i}|\nabla u_i|^2dx,\quad i=1,2.
  \]
  Clearly,
    \[
  K_{\delta_n,1}\ast{u}\rightarrow u ,\quad \mbox{in } L^2(\Gamma),
  \]
  as $n\rightarrow 0$.
  We are left to pass to the limit in 
  \[
  \frac{1}{\delta_n}
\int_\Gamma\left|K_{\delta_n,1}\ast u-K_{\delta_n,2}\ast u\right|^2dS.
  \]
  In fact, for any $x\in\Gamma$, by Taylor's expansion,
  \begin{align*}
    \left|K_{\delta_n,i}\ast u(x) -u(x)\right|
    &=\left|\frac{1}{\int_{\Omega_i}K_{\delta_n}(x-y)dy} \int_{\Omega_i}K_{\delta_n}(x-y)(u(y)-u(x))dy\right|\\
    &= 
    \left|\frac{1}{\int_{\Omega_i}K\left(\frac{|x-y|^2}{\delta_n^2}\right) dy} \int_{\Omega_i}K\left(\frac{|x-y|^2}{\delta_n^2}\right)
    \nabla u(z)\cdot (y-x)dy\right|
    \leq C \delta_n.
  \end{align*}
Consequently,
  \[
  \lim_{n\rightarrow \infty}\frac{1}{\delta_n}
\int_\Gamma\left|K_{\delta_n,1}\ast u-K_{\delta_n,2}\ast u\right|^2dS
=0.
  \]  
  The proof is completed by collecting the above results and taking $\limsup$ for $\mathcal{E}_{\delta_n} (u)$.
\end{proof}

To prove the convergence of minimizers for $\mathcal{E}_{\delta_n}$, a variant of the nonlocal Poincar\'{e}'s inequality is needed, in which the constant depends not on $\delta$.
\begin{lemma}\label{le:poincare_delta}
Let $\delta>0$, $1<p<\infty$, and $u\in L^p(\Omega)$ with $\int_\Omega u\,dx=0$.
There exists a constant $C$, depending only on $R$, $\Omega_1$, $\Omega_2$, and $p$, such that
\[
C\int_\Omega |u(x)|^pdx
\leq
 \frac{1}{\delta^p}\sum_{i=1,2}\int_{\Omega_i}\int_{\Omega_i}R_\delta(x-y)|u(y)-u(x)|^pdydx
 + \int_\Gamma |R_{\delta,1}\ast u-R_{\delta,2}\ast u|^pdS.
\]
\end{lemma}

\begin{proof}
The proof is similar to Lemma \ref{le:poincare}.
We argue by contradiction.
Assume that the inequality is false.
Then there exist $u_n\in L^2(\Omega)$ and a null sequence $\delta_n$, such that
\[
\|u_n\|_{L^p(\Omega)}=1,
\qquad
\int_\Omega u_n\,dx=0,
\]
and
\[
 \frac{1}{\delta_n^p}\sum_{i=1,2}\int_{\Omega_i}\int_{\Omega_i}R_{\delta_n}(x-y)|u_n(y)-u_n(x)|^pdydx
 + \int_\Gamma |R_{\delta_n,1}\ast u_n-R_{\delta_n,2}\ast u_n|^pdS
\leq \frac{1}{n}.
\]
We obtain, up to a subsequence, $u_n \rightharpoonup u$ in $L^p(\Omega)$ for a function $u\in L^p(\Omega)$.
By Proposition \ref{pr:2.4},
we have $u\in W^{1,p}(\Omega_1)\cap W^{1,p}(\Omega_2)$ and $\nabla u=0$ in $\Omega_i$.
Namely, $u$ is a constant in both $\Omega_1$ and $\Omega_2$.
Since $\int_\Gamma |R_{\delta_n,1}\ast u_n-R_{\delta_n,2}\ast u_n|^pdS \rightarrow 0$ as $n\rightarrow\infty$, we see that $u$ is a constant in $\Omega$.
This contradicts $\|u_n\|_{L^p(\Omega)}=1$ and $\int_\Omega u_ndx=0$.
\end{proof}

\begin{theorem}\label{th:converge_delta}
Assume that $\Omega_1$ and $\Omega_2$ are smooth.
Let $\lambda_1, \lambda_2>0$, $f\in L^2(\Omega)$, and $g\in L^2(\Gamma)$. If $\delta_n$ is a null sequence,
   $u_n$ and $u$ are the minimizers of $\mathcal{E}_{\delta_n}$ and $\mathcal{E}$, then
   \[
   u_n\rightarrow u,\quad \mbox{in } L^2(\Omega),
   \]
   as $n\rightarrow\infty$.
\end{theorem}

\begin{proof}
Since $u_n$ is a minimizer of $\mathcal{E}_{\delta_n}$, we have $\sup_{n}|\mathcal{E}_{\delta_n}(u_n)|<\infty$.
Notice that
\begin{align*}
  \frac{1}{\delta_n^2}\int_{\Omega}\int_{\Omega}\Lambda(x,y)R_{\delta_n}(x-y)|u_n(y)-u_n(x)|^2dydx
  +\int_\Gamma\left|K_{\delta_n,1}\ast u_n-K_{\delta_n,2}\ast u_n\right|^2dS
  \leq \mathcal{E}_{\delta_n}(u_n) \\
  +\int_{\Omega}fu_ndx
  +\int_\Gamma g K_{\delta_n,1}\ast u_n dS.
\end{align*}
We apply Lemma \ref{le:poincare_delta} for the left-hand side and apply Cauchy's inequality, the trace inequality, \eqref{eq:Ku:4} for the right-hand side to obtain that $\int_\Omega |u_n|^2dx$ is uniformly bounded.
Proposition \ref{pr:2.4} then implies that $u$ is precompact in both $L^2(\Omega_1)$ and $L^2(\Omega_2)$.
Consequently, $u$ is precompact in $L^2(\Omega)$.
  The proof is completed by utilizing Theorem \ref{th:Gamma_converge_delta} and Proposition \ref{pr:2.2}.
\end{proof}

We state the $\Gamma$-convergence and the convergence of minimizers for $\mathcal{F}_{\delta_n}$ as follows. The proof is similar to the case of $\mathcal{E}_{\delta_n}$ and is omitted here.
\begin{theorem}
  Let $1<p<\infty$, $q=\frac{p}{p-1}$, $\lambda_1, \lambda_2>0$, $f\in L^q(\Omega)$, and $g\in L^q(\Gamma)$. Then for any null sequence $\delta_n$,
  \begin{equation*}
    \mathcal{F}_{\delta_n}\stackrel{\Gamma}{\longrightarrow }\mathcal{F},
  \end{equation*}
  in $L^p(\Omega)$ as $n\rightarrow\infty$.

  Furthermore, assume that $\Omega_1$ and $\Omega_2$ are smooth. If $u_n$ and $u$ are the minimizers of $\mathcal{F}_{\delta_n}$ and $\mathcal{F}$, then
   \[
   u_n\rightarrow u,\quad \mbox{in } L^p(\Omega),
   \]
   as $n\rightarrow\infty$.
\end{theorem}

\section{Numerical validation}
We numerically verify the convergence of the nonlocal minimizer to its local counterpart by measuring the $L^2$ error between $u_n$ and $u$, the minimizers of $\mathcal{E}_{\delta_n}$ and $\mathcal{E}$. One- and two-dimensional problems are tested; higher dimensions are analogous. In the sequel we omit the subscript and write $\delta\to0$. The code is available at \url{https://github.com/shwangtangjun/Nonlocal-Interface}.

\subsection{Numerical method}
The method is adopted from \cite{gan2025nonlocal}, where multidimensional integrals are evaluated efficiently by exploiting the tensor-product structure of the Gaussian kernel, the domain, and the finite element basis.

We use the Gaussian kernel
\[
  R_\delta(x-y)=K_\delta(x-y)=c_\delta
  \exp\left(-\frac{|x-y|^2}{\delta^2}\right),
\]
where $c_\delta=2 \pi^{-d/2}\delta^{-d}$ satisfies the normalization of $R$. The same kernel is used for $K$; its multiplicative constant cancels in the normalized convolution $K_{\delta,i}\ast u$.

The computational domain is $\Omega=[0,1]^d$, partitioned into a uniform tensor-product mesh
$$\Omega=\bigcup_{i_1,\cdots,i_d=0}^{N-1} \Omega_{i_1,\cdots,i_d},$$
with
$$\Omega_{i_1,\cdots,i_d}=e_{i_1}\times \cdots \times e_{i_d},\quad e_i=[ih,(i+1)h],\quad h=1/N.$$ 
The finite element space is
\[V_h=\{u\in C(\Omega): u\; \mbox{is multi-cubic polynomial on} \;\Omega_{i_1,\cdots,i_d}\}.\]
We solve $\min_{u\in V_h}\mathcal{E}_\delta(u)$. The standard piecewise cubic Lagrange basis is associated with the nodes $s_j=jh/3$, $j=0,\ldots,3N$. The one-dimensional basis functions $\{\psi_j\}_{j=0}^{3N}$ are
\[
\psi_{3i}(s)=
\begin{cases}
\displaystyle
\frac{(s-s_{3i-3})(s-s_{3i-2})(s-s_{3i-1})}
{(s_{3i}-s_{3i-3})(s_{3i}-s_{3i-2})(s_{3i}-s_{3i-1})},
& s\in e_{i-1},\\[2mm]
\displaystyle
\frac{(s-s_{3i+1})(s-s_{3i+2})(s-s_{3i+3})}
{(s_{3i}-s_{3i+1})(s_{3i}-s_{3i+2})(s_{3i}-s_{3i+3})},
& s\in e_i,\\[2mm]
0, & \mbox{otherwise},
\end{cases}
\]
\[
\psi_{3i+1}(s)=
\begin{cases}
\displaystyle
\frac{(s-s_{3i})(s-s_{3i+2})(s-s_{3i+3})}
{(s_{3i+1}-s_{3i})(s_{3i+1}-s_{3i+2})(s_{3i+1}-s_{3i+3})},
& s\in e_i,\\[2mm]
0, & \mbox{otherwise},
\end{cases}
\]
\[
\psi_{3i+2}(s)=
\begin{cases}
\displaystyle
\frac{(s-s_{3i})(s-s_{3i+1})(s-s_{3i+3})}
{(s_{3i+2}-s_{3i})(s_{3i+2}-s_{3i+1})(s_{3i+2}-s_{3i+3})},
& s\in e_i,\\[2mm]
0, & \mbox{otherwise}.
\end{cases}
\]
In two dimensions we use the tensor-product basis
\[
  \psi_{j_1,j_2}(x_1,x_2)=\psi_{j_1}(x_1)\psi_{j_2}(x_2),
  \qquad j_1,j_2=0,\ldots,3N.
\]
Hence, for $d=1,2$,
\[
  u(x)=\sum_{j=0}^{3N}U_j\psi_j(x),
  \qquad
  u(x_1,x_2)=\sum_{j_1,j_2=0}^{3N}
  U_{j_1,j_2}\psi_{j_1}(x_1)\psi_{j_2}(x_2).
\]

Kernel integrals of the form
\[
  \int_a^b \exp\left(-\frac{(x-y)^2}{\delta^2}\right)\psi_j(y)\,dy
\]
are evaluated via the change of variables $r=y-x$, $\eta=1/\delta$, reducing to linear combinations of the Gaussian moments
\begin{align*}
f_0(\eta,a,b)&=\int_a^b e^{-\eta^2r^2}\,dr
=\frac{\sqrt{\pi}}{2\eta}\left(\operatorname{erf}(\eta b)-\operatorname{erf}(\eta a)\right),\\
f_1(\eta,a,b)&=\int_a^b r e^{-\eta^2r^2}\,dr
=\frac{1}{2\eta^2}\left(e^{-\eta^2a^2}-e^{-\eta^2b^2}\right),\\
f_2(\eta,a,b)&=\int_a^b r^2 e^{-\eta^2r^2}\,dr=\frac{a}{2\eta^2}e^{-\eta^2a^2}
-\frac{b}{2\eta^2}e^{-\eta^2b^2}
+\frac{1}{2\eta^2}f_0(\eta,a,b),\\
f_3(\eta,a,b)&=\int_a^b r^3 e^{-\eta^2r^2}\,dr=\frac{a^2}{2\eta^2}e^{-\eta^2a^2}
-\frac{b^2}{2\eta^2}e^{-\eta^2b^2}
+\frac{1}{\eta^2}f_1(\eta,a,b).
\end{align*}
Here $\operatorname{erf}(z)=\frac{2}{\sqrt{\pi}}\int_0^z e^{-t^2}\,dt$. Since each $\psi_j$ is a cubic polynomial, every one-dimensional kernel coefficient is expressed from $f_0,f_1,f_2,f_3$.
In two dimensions, tensor-product structure yields
\begin{align*}
&\int_{e_{i_1}\times e_{i_2}}
\exp\left(-\frac{|x-y|^2}{\delta^2}\right)
\psi_{j_1}(y_1)\psi_{j_2}(y_2)\,dy\\
&\quad =
\int_{e_{i_1}}\exp\left(-\frac{(x_1-y_1)^2}{\delta^2}\right)\psi_{j_1}(y_1)\,dy_1
\int_{e_{i_2}}\exp\left(-\frac{(x_2-y_2)^2}{\delta^2}\right)\psi_{j_2}(y_2)\,dy_2.
\end{align*}
Thus multidimensional kernel integrations reduce to products of one-dimensional quantities.

Sharp interface integrals are treated analogously: in 1D, $\Gamma$ is a single point; in 2D, it is a union of line segments, and the integral is the sum over each segment.

We employ the composite Simpson's $3/8$ rule on each cell for $\int_\Omega f u\,dx$ and for the 2D interface integrals; the integrands are sufficiently regular. For the nonlocal energy term, which approximates $|\nabla u(x)|^2$, the integrand exhibits an $O(\delta)$ boundary layer at cell interfaces because the cubic basis is only $C^0$ across elements. Hence we use a tailored quadrature on each one-dimensional cell $e_i=[ih,(i+1)h]$: if $6\delta<h$, the cell is split into
\[
[ih,ih+3\delta],\quad [ih+3\delta,(i+1)h-3\delta],\quad [(i+1)h-3\delta,(i+1)h],
\]
and a 10-point Gauss-Legendre rule is applied on each subinterval; if $6\delta\ge h$, a 30-point Gauss-Legendre rule is used on the whole cell. The two-dimensional quadrature is obtained by tensor products.

After assembly, $\mathcal{E}_\delta(u)$ is a quadratic form in the coefficient vector $U\in\mathbb{R}^{(3N+1)^d}$; differentiation gives $SU=b$. The zero-mean constraint $\int_\Omega u\,dx=0$ is imposed via a Lagrange multiplier. Let $m$ be the vector of Simpson weights. The augmented system
\[
  \begin{bmatrix}
  S & m\\
  m^T & 0
  \end{bmatrix}
  \begin{bmatrix}
  U\\ \mu
  \end{bmatrix}
  =
  \begin{bmatrix}
  b\\0
  \end{bmatrix}
\]
is solved by a direct sparse solver.

\subsection{Experiment setup}
We set $\lambda_1=1$, $\lambda_2=9$. Given a ground-truth solution $u_{\mathrm{gt}}$, we take
\[
  f=-\lambda_i\Delta u_{\mathrm{gt}}\quad \mbox{in } \Omega_i,\qquad i=1,2,
\]
and compute $g$ from the flux jump across $\Gamma$. All $u_{\mathrm{gt}}$ have zero mean, consistent with the Neumann constraint.

In 1D, $\Omega=(0,1)$, $\Omega_1=(0,1/2)$, $\Omega_2=(1/2,1)$, and
\[
u_{\mathrm{gt}}(x)=\cos(\pi x).
\]

In 2D, $\Omega=(0,1)^2$, $\Omega_1=(0.3,0.6)\times(0.4,0.7)$, $\Omega_2=\Omega\setminus\overline{\Omega}_1$.
Let
\[
q(t)=3t^2-2t^3-\frac{1}{2},\qquad \alpha=0.35,\qquad \beta=0.25.
\]
Then
\[
u_{\mathrm{gt}}(x,y)=q(x)+\alpha q(y)+\beta q(x)q(y).
\]

We use $N=10,20,30,40,50$ and 50 logarithmically spaced $\delta\in[10^{-5},10^{-2}]$. The numerical solution is interpolated onto $M=5000$ uniformly distributed points in 1D and an $M\times M$ grid with $M=1000$ in 2D. Errors are defined as
\begin{align*}
\operatorname{Error}_{\text{1D}}
&=\left(\frac{1}{M}\sum_{j=0}^{M-1}|u(x_j)-u_{\mathrm{gt}}(x_j)|^2\right)^{1/2},\\
\operatorname{Error}_{\text{2D}}
&=\left(\frac{1}{M^2}\sum_{j_1=0}^{M-1}\sum_{j_2=0}^{M-1}
|u(x_{j_1},y_{j_2})-u_{\mathrm{gt}}(x_{j_1},y_{j_2})|^2\right)^{1/2}.
\end{align*}

\subsection{Numerical results}
Figure~\ref{fig:numerical_1d} shows the 1D convergence. In the log-log scale the $L^2$ error decays approximately linearly with $\delta$ for all $N$. The curves nearly coincide, and the inset reveals only a mild dependence on $N$, indicating that the nonlocal approximation error dominates.

\begin{figure}[H]
\centering
\includegraphics[width=0.78\linewidth]{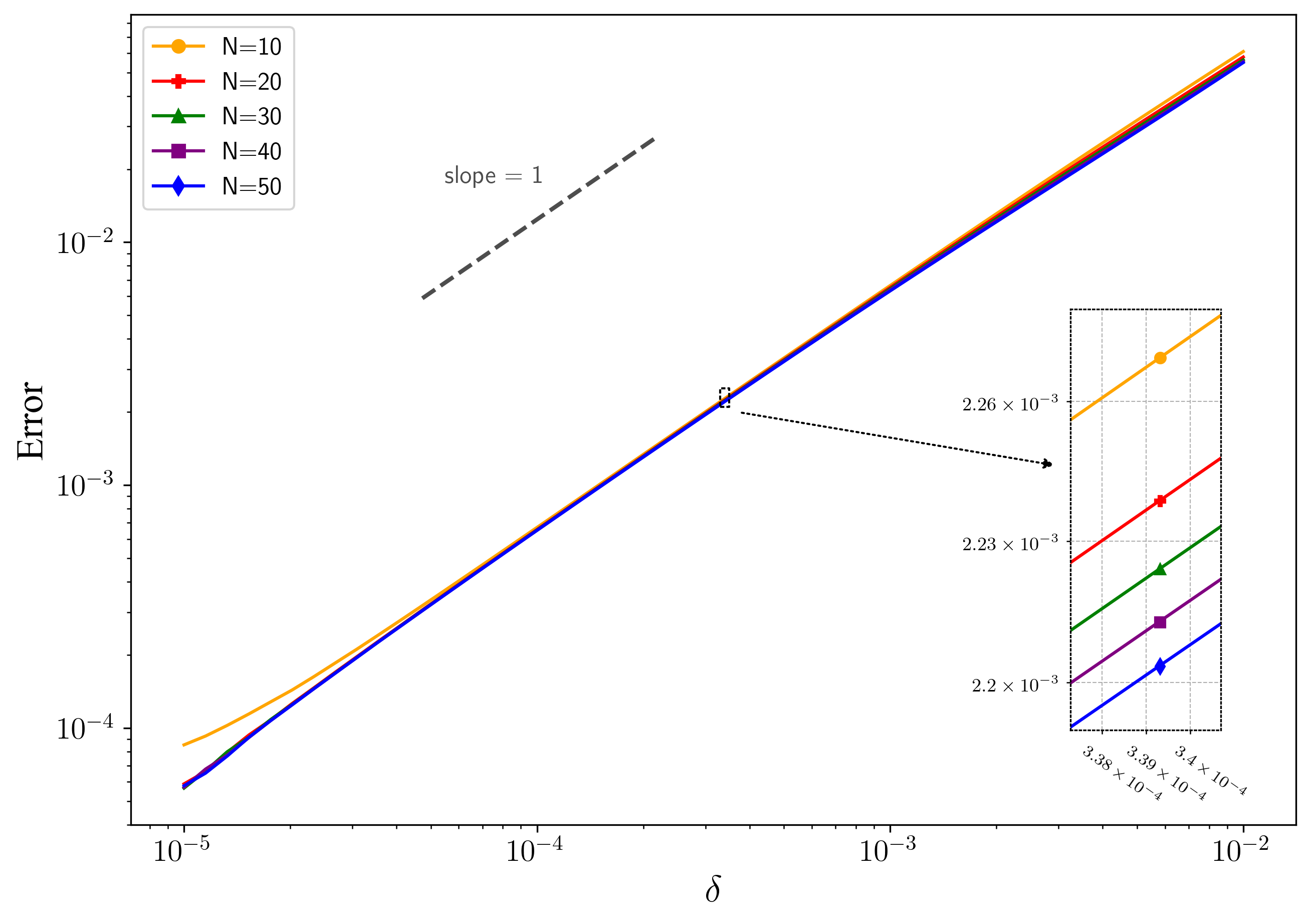}
\caption{The $L^2$ error $\|u-u_{\mathrm{gt}}\|_{L^2}$ for the one-dimensional interface problem.}
\label{fig:numerical_1d}
\end{figure}

Figure~\ref{fig:numerical_2d} displays the 2D results; the same first-order trend with respect to $\delta$ is observed. Although the interface is a closed rectangle, the behavior remains consistent with the 1D case. These experiments support convergence of the sharp-interface nonlocal model as $\delta\to0$.

\begin{figure}[H]
\centering
\includegraphics[width=0.78\linewidth]{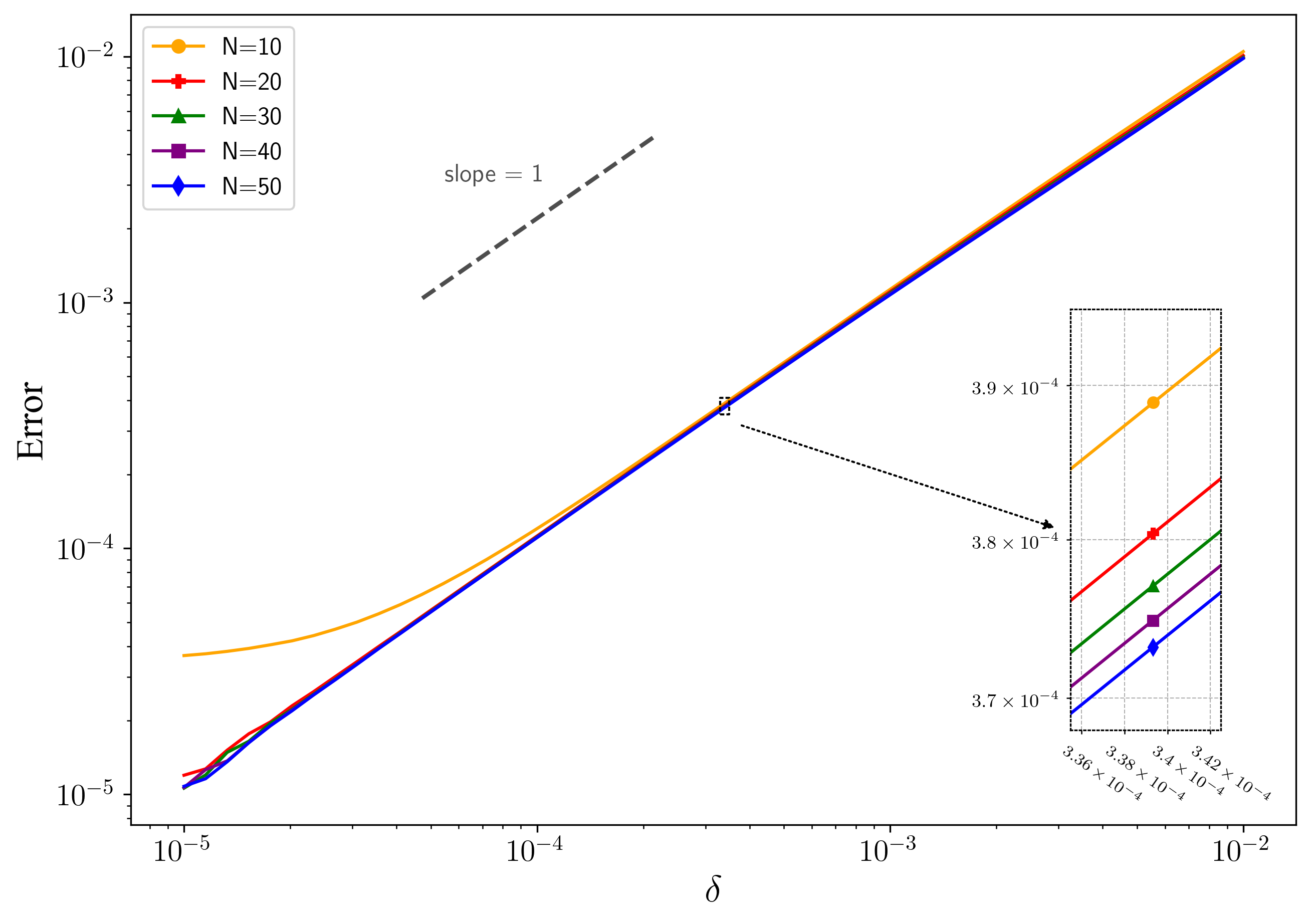}
\caption{The $L^2$ error $\|u-u_{\mathrm{gt}}\|_{L^2}$ for the two-dimensional interface problem.}
\label{fig:numerical_2d}
\end{figure}

\section{Conclusion}
We proposed an energy-based nonlocal Poisson interface problem that preserves a sharp interface.
This formulation yields several favorable properties and readily extends to more general interface problems, which we verified through a $p$-Laplacian interface problem with membrane conditions.
Using $\Gamma$-convergence, we proved that minimizers of the nonlocal functionals converge to those of the corresponding local problems.
Numerical experiments further confirm this convergence.




\bibliographystyle{unsrt}
\bibliography{references}
\end{document}